\documentclass%
[12pt]%
{article}
\usepackage{amsmath}
\usepackage{amsfonts}
\usepackage{amssymb}
\usepackage{amsthm}
\usepackage{xypic}
\usepackage{xparse}
\usepackage{setspace}
\usepackage{aliascnt}
\usepackage{hyperref}
\title{Simple commutative algebras in Deligne's categories $\rep{t}$}
\author{Luke Sciarappa}
\newcommand{\h}{\mathfrak{h}}
\newcommand{\Rep}[1]{\mathbf{Rep}(S_{#1})}
\newcommand{\Gal}{\operatorname{Gal}}

\newcommand{\Alg}{\operatorname{Alg}}

\newcommand{\Ind}[2]{\operatorname{Ind}_{S_{#1}\times S_{#2 - #1}}^{S_{#2}}}
\newcommand{\ind}[1]{{}^\uparrow#1}
\newcommand{\Res}[2]{\operatorname{Res}_{S_{#1}\times S_{#2 - #1}}^{S_{#2}}}

\newcommand{\triv}{\mathbb{C}}
\newcommand{\bigness}{level}
\newcommand{\Hom}{\mathrm{Hom}}

\newcommand{\parti}
[2]{P_{#1 + #2}}
\newcommand{\C}{\mathbb{C}}
\newcommand{\A}{\mathbb{A}_{\C}^1}
\newcommand{\Spec}{\operatorname{Spec}}
\newcommand{\ob}{\operatorname{Ob}}
\newcommand{\mor}{\operatorname{Mor}}

\newcommand{\mon}[1]{\mathbf{MonAbCat}_{#1}}
\newcommand{\cat}[1]{\mathbf{Cat}_{#1}}

\newtheorem{thm}{Theorem}[section]

\newaliascnt{lem}{thm}
\newtheorem{lem}[lem]{Proposition}
\aliascntresetthe{lem}

\newtheorem{cor}{Corollary}[thm]

\theoremstyle{definition}
\newtheorem{defn}[thm]{Definition}

\begin{document}

\DeclareDocumentCommand\repO{ m g }{%
    {\mathrm{Rep}_0(S_{#1}%
        \IfNoValueF {#2} { , #2}%
    )}%
}
\DeclareDocumentCommand\rep{ m g }{%
    {\mathrm{Rep}(S_{#1}%
        \IfNoValueF {#2} { , #2}%
    )}%
}
\DeclareDocumentCommand\Rep{ m g }{%
    {\mathbf{Rep}(S_{#1}%
        \IfNoValueF {#2} { , #2}%
    )}%
}
\maketitle

\begin{abstract}
We show that in the Deligne categories $\rep{t}$ for $t$ a transcendental number, the only simple algebra objects are images of simple algebras in the category of representations of a symmetric group under a canonical induction functor. They come in families which interpolate the families of algebras of functions on the cosets of $H\times S_{n-k}$ in $S_n$, for a fixed subgroup $H$ of $S_k$.
\end{abstract}
\section{Introduction}
The categories $\rep{t}$, defined for $t\in\C$, which interpolate the categories $\Rep{n}$ of representations of the symmetric group, were introduced by P.~Deligne \cite{del}. In addition to generalizing the representation theory of the symmetric groups, they were among the first symmetric $\C$-linear tensor categories known that did not admit a fiber functor to the category of complex vector spaces or supervector spaces.

A symmetric monoidal structure is exactly what is needed on a category for it to have a good notion of commutative algebra object, and the additional structure of a symmetric tensor category --- the $\C$-linear structure in particular --- makes it a good setting for commutative algebra. This theory of commutative algebra can, dually, be thought of as a theory of algebraic geometry internal to the given category; for example, in the symmetric tensor category $\mathbf{Rep}(G)$ of complex representations of a group, the internal theory of algebraic geometry amounts to, more or less, the study of schemes over $\C$ with a $G$-action. 

In light of this, it is natural to study the commutative (unital, associative) algebra objects in $\rep{t}$. The most basic question in a theory of commutative algebra is to describe the simple algebras. This is known in the case of $\Rep{n}$; simple algebras correspond to transitive actions. However, in the case of $\rep{t}$, internal algebras are no longer ordinary algebras with extra structure; the forgetful fiber functor from $\Rep{n}$ to vector spaces lets us consider algebras in $\Rep{n}$ as algebras in the category of vector spaces equipped with a compatible $S_n$-action, but such a functor does not exist for $\rep{t}$. 

Using induction functors, given a simple algebra $A$ in $\Rep{n}$, one can define a family of algebras $\ind{A}(t)\in\rep{t}$ which are simple for almost all $t$. 
We show that when $t$ is transcendental, all simple algebras in $\rep{t}$ have this form. That is, all the simple algebras in $\rep{t}$ for transcendental $t$ arise as algebras coming from $\Rep{k}$ for some $k$. Simple algebras in $\Rep{k}$ have the form $\C[S_k/H]$ for $H$ a subgroup of $S_k$, as shown in \autoref{app}. Intuitively, the simple algebras in $\rep{t}$ can be thought of as \mbox{``$\C[S_t/H\times S_{t-k}]$''}, for $H$ a subgroup of $S_k$; rigorously, they come in families $\ind{A}(t)$ such that for sufficiently large $n$, the projection to $\Rep{n}$ of $\ind{A}(n)\in\rep{n}$ is $\C[S_n/H\times S_{n-k}]$ for a fixed $H\subset S_k$. 

In \autoref{background}, we summarize background information about $\rep{t}$. In \autoref{cats-over}, we discuss how $\rep{t}$ and related categories can be considered as categories internal to schemes over $\A$. In \autoref{cons}, we establish the results we will need pertaining to constructible sets. In \autoref{sec:main}, we prove the main theorem, classifying simple algebras in $\rep{t}$ for transcendental $t$. In \autoref{app}, we establish some facts about the groups $S_n$ and categories $\Rep{n}$ which are necessary for the arguments about $\rep{t}$.

\section*{Acknowledgments}

The author thanks Akhil Mathew, who first introduced the author to the category $\rep{t}$. The author thanks Professor Pavel Etingof, who suggested this direction of investigation. The author greatly thanks Nathan Harman, for many helpful suggestions and conversations about these ideas.
 Additionally, the author gladly thanks MIT PRIMES, for enabling this research, and in particular Dr.~Tanya Khovanova for helpful advice about mathematical writing.

\section{Background}\label{background}
By an \emph{algebra}, we always mean an associative, commutative, unital algebra, unless explicitly denied.
The disjoint union of two finite sets $a$, $b$ is denoted $a+b$, and the natural maps are denoted $\iota_0 : a \hookrightarrow a+b$ and $\iota_1 : b \hookrightarrow a+b$.
The set of partitions of a set $a$ is denoted $P_a$. For $n\in\mathbb N$, there is a canonical finite set $[n] = \{k\in\mathbb N|1\leq k\leq n\}$ with $n$ elements.
\begin{defn}[\cite{cos}]
For $R$ a commutative ring and $t\in R$, the $R$-linear monoidal category $\repO{t}{R}$ is defined to have finite sets as objects and $\Hom(a,b) = R[\parti{a}{b}]$.
\begin{itemize}
\item The identity element of $\Hom(a,a)$ is the partition whose parts are $\{\iota_0x, \iota_1x\}$ for all $x \in a$. 
\item Composition is the $R$-bilinear extension of a function \[-\circ- : \parti{b}{c}\times\parti{a}{b} \to R[\parti{a}{c}],\] defined as follows. Given $\rho \in \parti{b}{c}$ and $\pi \in \parti{a}{b}$, let $\rho\vee\pi$ be the finest partition on $a + b + c$ containing $\rho$ and $\pi$. Letting $\rho\cdot\pi$ denote the restriction of $\rho\vee\pi$ to $a + c$, and letting $d(\rho,\pi)$ denote the number of parts of $\rho\vee\pi$ containing only elements of $b$, we define $\rho\circ\pi$ to be  $t^{d(\rho,\pi)}\rho\cdot\pi$. 
\item The tensor product of two objects is their disjoint union, with unit $[0]$. On morphisms, the tensor product is the $R$-bilinear extension of a function \[-\otimes- : \parti{a}{b}\times\parti{a'}{b'}\to \parti{a+a'}{b+b'}.\] The partition $\pi\otimes\pi'$ of $a+a'+b+b'$ partitions $a+b$ as $\pi$ does and $a'+b'$ as $\pi'$ does, and is finest among such. 
\end{itemize}
\end{defn}
\begin{defn}[\cite{cos}]
The category $\rep{t}{R}$ is defined to be the category obtained by starting with $\repO{t}{R}$, then adjoining direct sums, then taking the idempotent completion (also called the Karoubification). We write $\rep{t}$ for $\rep{t}{\mathbb{C}}$, etc. 
\end{defn}Mostly, we will be concerned with the case when $R$ is the complex numbers $\C$.
Certain objects in $\rep{t}$ get special names: the monoidal unit $([0], \text{id}_{[0]})$ will be denoted $\triv$ and the generating object $([1], \text{id}_{[1]})$ will be denoted $\h$. These will also be used to denote the corresponding objects in $\Rep{k}$, namely $\mathbb{C}$ with the trivial action and $\mathbb{C}^{k}$ with the action of permuting coordinates, respectively.

\begin{defn}
Given a map $\varphi : R \to R'$, there is a family of monoidal functors $\varphi_*: \rep{t}{R} \to \rep{\varphi(t)}{R'}$ indexed by $t\in R$. (We suppress $t$ in the notation.)  Specifically, we define $\varphi_{*0}: \repO{t}{R}\to\repO{\varphi(t)}{R'}$ to be the identity on objects and send a morphism $f = \sum_{\pi\in\parti{a}{b}} f_\pi \pi : a \to b$ to $\sum_{\pi\in\parti{a}{b}}\varphi(f_\pi)\pi : a \to b$. The functor $\varphi_*$ is then defined to be the extension of $\varphi_{*0}$ to $\rep{t}{R}$ which comes from the universal properties of completion under direct sums and idempotents.
\end{defn}
The functors $\varphi_*$ are ``$\varphi$-linear'', in the sense that for morphisms $f$, $g$ and $r \in R$, the equation $\varphi_*(r f + g) = \varphi(r)\varphi_*(f)+\varphi_*(g)$ holds. In particular, even if $R = R'$, they are generally not $R$-linear, only linear over the subring of elements fixed by $\varphi$. As a special case, if $\theta \in \Gal(\mathbb{C}/\overline{\mathbb{Q}})$, then $\theta_*$ is a (monoidal) equivalence of $\overline{\mathbb{Q}}$-linear categories, since $(\theta^{-1})_*$ is inverse to $\theta_*$. It follows from this that if $X$ is an algebra in $\rep{t}$, the induced algebra structure on $\theta_*X\in\rep{\theta(t)}$ is simple if and only if $X$ is. 

In general, we call the $\varphi_*$ \emph{base change} functors. If they come from an automorphism of $\mathbb{C}$ over $\overline{\mathbb{Q}}$, we may also call them \emph{Galois functors}.
\begin{defn}
For two $R$-linear categories $\mathcal A$, $\mathcal B$ the \emph{external tensor product} $\mathcal{A}\boxtimes\mathcal{B}$ is generated under direct sums by objects $A \boxtimes B$ for $A \in \mathcal A$, $B \in \mathcal B$, with $\Hom_{\mathcal A\boxtimes\mathcal B}(A \boxtimes B, A'\boxtimes B') = \Hom_{\mathcal{A}}(A,A')\otimes_R\Hom_{\mathcal{B}}(B,B')$. If $\mathcal A$ and $\mathcal B$ are monoidal categories, $\mathcal A\boxtimes\mathcal B$ inherits a monoidal structure with $(A \boxtimes B) \otimes (A' \boxtimes B') = (A \otimes A') \boxtimes (B \otimes B')$, etc. 
\end{defn}

\begin{thm}[\cite{del}]
For all $R$-linear monoidal categories $\mathcal T$, the category of $R$-linear monoidal functors $\rep{t}{R}\to\mathcal T$ is equivalent to the category of commutative Frobenius algebras of dimension $t$ in $\mathcal T$.
\end{thm}
\begin{proof}
 One half of this equivalence sends a monoidal functor $F$ to $F(\h)$ with the induced Frobenius algebra structure. For the other half, given a commutative Frobenius algebra $T$, the multiplication and comultiplication can be assembled into maps $\phi_{a,b} : T^{\otimes a} \to T \to T^{\otimes b}$. A functor $\repO{t}{R}\to \mathcal T$ can be defined by sending $a$ to $T^{\otimes a}$ and a partition $\pi : a \to b$ to $\bigotimes_p \phi_{a\cap p,b\cap p}$, where $p$ ranges through parts of $\pi$. Since all linear functors preserve direct sums and images of idempotents, the extension to $\rep{t}{R}$ is essentially uniquely determined. For details, see \cite{del}.
\end{proof}
This universal property allows the definition of restriction and induction functors.
\begin{defn}[\cite{eti}]
The \emph{restriction functor} $\Res{k}{t} : \rep{t} \to \Rep{k}\boxtimes\rep{t-k}$ is defined to be the monoidal functor that sends $\h\in\rep{t}$ to the Frobenius algebra $\h\boxtimes \triv \oplus \triv\boxtimes\h$. The \emph{induction functor} $\Ind{k}{t} : \Rep{k}\boxtimes\rep{t-k}\to\rep{t}$ is defined to be the left adjoint to restriction.
\end{defn}
\begin{lem}\label{comm-lem}
 Induction and restriction commute, in an appropriate sense, with the Galois functors. Specifically, 
\begin{equation}
\theta_* \Ind{k}{t} \cong \Ind{k}{\theta(t)} (1\boxtimes\theta_*)
\end{equation}
as functors $\Rep{k}\boxtimes\rep{t-k}\to\rep{\theta(t)}$, and
 \begin{equation}
\Res{k}{\theta(t)}\theta_*\cong (1\boxtimes \theta_*) \Res{k}{t}
\end{equation}
as functors $\rep{t} \to \Rep{k}\boxtimes\rep{\theta(t)-k}$.
\end{lem}
\begin{proof}
\sloppy
This can be seen using the universal property of $\rep{t}$. The  $\mathbb{C}$-linear monoidal functors $(1\boxtimes\theta_*)(\Res{k}{t})(\theta_*^{-1})$ and  $\Res{k}{\theta(t)}: \rep{\theta(t)} \to \Rep{k}\boxtimes\rep{\theta(t)-k}$, when evaluated on the canonical commutative Frobenius algebra $\h$ in $\rep{\theta(t)}$, give isomorphic algebras in $\Rep{k}\boxtimes\rep{\theta(t)-k}$. 
Therefore the functors are isomorphic, and a categorical argument using the Yoneda lemma gives the corresponding statement for induction:
\begin{align*}
&\Hom_{\rep{\theta(t)}}(\Ind{k}{\theta(t)}(1 \boxtimes \theta_*)A, B)\\\cong\:
&\Hom_{\Rep{k}\boxtimes\rep{\theta(t)-k}}((1 \boxtimes \theta_*)A,\Res{k}{\theta(t)}B)\\\cong\:
&\Hom_{\Rep{k}\boxtimes\rep{t-k}}(A,(1\boxtimes\theta_*^{-1})\Res{k}{\theta(t)}B)\\\cong\:
&\Hom_{\Rep{k}\boxtimes\rep{t-k}}(A,\Res{k}{t}\theta_*^{-1}B)\\\cong\: 
&\Hom_{\rep{t}}(\Ind{k}{t}A,\theta_*^{-1}B)\\\cong\:
&\Hom_{\rep{\theta(t)}}(\theta_*\Ind{k}{t}A,B).\qedhere
\end{align*}
\end{proof}
\fussy
$\Rep{k}$ is a semisimple category, and the simple objects, which are the irreducible representations, are naturally labeled by integer partitions of $k$. An \emph{integer partition} $\lambda$ is an eventually zero nonincreasing sequence of natural numbers. It is a partition `of $k$' if the sum of the sequence is $k$. In this case, $k$ is also called the \emph{size} of $\lambda$ and we write $k = |\lambda|$. These properties, suitably modified, also hold of $\rep{t}$, as seen in the following.
\begin{thm}[\cite{cos}]
The indecomposable objects of $\rep{t}{F}$ are naturally labeled by integer partitions of arbitrary size. These objects are simple, and $\rep{t}{F}$ is semisimple, when $t$ is not a nonnegative integer.
\end{thm}
The object corresponding to a partition $\lambda$ will be denoted $L(\lambda)$.

Since $\rep{t}$ is supposed to interpolate the categories $\Rep{k}$, it may seem surprising that $\rep{k}$ is not even semisimple. To pin down the precise relation that $\rep{k}$ bears to $\Rep{k}$, we introduce the notion of \bigness{}.
\begin{defn}
The \emph\bigness{} of an indecomposable object $L(\lambda)$ is defined to be the size $|\lambda|$ of $\lambda$. The \emph\bigness{} of any object is defined to be the maximum \bigness{} of any indecomposable summand of that object.
\end{defn}
\begin{thm}[{\cite[section 5]{del}}]\label{level-equiv}
For a natural number $k$ and an indecomposable object $X$, the following are equivalent: 
\begin{itemize}
\item The \bigness{} of $X$ is $k$. 
\item $k$ is the smallest natural number such that there is a monomorphism $X \hookrightarrow \Ind{k}{t}(B\boxtimes\triv)$ for some $B \in \Rep{k}$. 
\item $k$ is the smallest natural number such that there is a monomorphism $X \hookrightarrow \h^{\otimes k}$. 
\end{itemize}
\end{thm}
From the above, we obtain the following.
\begin{cor}\label{level-cor} For any objects $X$ and $Y$, if $X$ has level $k$, and $Y$ has level $\ell$, then $X \otimes Y$ has level $k + \ell$. Also, for $B \in \Rep{k}$, if $X$ has level $\ell$, then $\Ind{k}{t} (B \boxtimes X)$ has level $k + \ell$. 
\end{cor}
We write $\rep{t}^{(k)}$ for the full subcategory of $\rep{t}$ consisting of objects of level $\leq k$. There are then functors $-\otimes- : \rep{t}^{(k)}\boxtimes\rep{t}^{(\ell)}\to\rep{t}^{(k+\ell)}$ and $\Ind{k}{t} : \Rep{k}\boxtimes\rep{t-k}^{(\ell)}\to\rep{t}^{(k+\ell)}$ which are restricted and corestricted versions of the tensor product and induction.

The notion of \bigness{} allows us to state precisely the relation between $\rep{n}$ and $\Rep{n}$.
\begin{thm}\label{proj-equiv}
There is a $\C$-linear monoidal functor $\rep{n} \to \Rep{n}$ which is an additive equivalence onto its image when restricted to $\rep{n}^{(k)}$ when $n > 2k$.
\end{thm}
\begin{proof}
By \cite[Proposition 3.25]{cos}, there is a functor $\rep{n}\to\Rep{n}$ which is full. This functor sends $L(\lambda)$ to $0$ only if $n-|\lambda|<\lambda_1$. ($\lambda$ is a sequence of natural numbers, and $\lambda_1$ is the first number in the sequence.) If $n-|\lambda|\geq\lambda_1$, then $L(\lambda)$ is sent to the irreducible representation of $S_n$ labeled by the partition $(n-|\lambda|, \lambda_1,\lambda_2,\ldots)$. If $n\geq2k$, all indecomposables in $\rep{n}^{(k)}$ have $k\geq|\lambda|$, so 
$n\geq 2 k \geq 2 |\lambda| $
and $n-|\lambda|\geq |\lambda|\geq \lambda_1$. So no indecomposables in $\rep{n}^{(k)}$ are sent to zero; in fact, by the above, they are all sent to distinct irreducibles in $\Rep{n}$. 
\end{proof}

\section{Categories Over $\A$}\label{cats-over}
Mathew \cite{mat} introduced the framework of categories parameterized by schemes, and showed that the categories $\rep{t}$ fit into this framework. That is, there is a category $\rep{*}$ internal to the category of schemes over $\mathbb{A}^1_{\mathbb{C}}$ such that the fiber over $t$ is equivalent to $\rep{t}$. It consists of an object scheme $\ob{}$, a morphism scheme $\mor$, and certain maps of schemes over $\A$ expressing domain, codomain, composition, tensor, etc.  The theorems and constructions of Deligne, and Ostrik and Comes, discussed above can be reformulated in this language. Since this is not the focus of our paper, we will use the language of Mathew without reproving the results formulated in the other language. If $\mathfrak C$ is a category over $\A$, we write $\mathcal R\mathfrak C$ for the category of global sections. If $X$ is a point in the object scheme of a category over $\A$, we call the element of $\A$ whose fiber $X$ occupies the \emph{rank} of $X$.

\begin{defn}
Following the notation of \cite{mat}, for all schemes $S$, there is a functor $\Alg$ from the category $\mon{S}$ of monoidal Ab-enriched categories over $S$ to the category $\cat{S}$ of categories over $S$ that sends $\mathfrak C$ to the category of algebras internal to $\mathfrak C$.
\end{defn} 
For example, if $M$ is the morphism scheme of $\mathfrak C$, the object scheme of $\Alg \mathfrak C$ is a certain subscheme of $M\times_S M$ which is the intersection (pullback) of the equalizers of several pairs of morphisms, some of which express that the domains and codomains are correct, others of which specify associativity and the other algebra laws.
\begin{defn}
Let $\Alg\rep{*}^{(k)}$ denote the pullback of \[\Alg\rep{*}\rightarrow\rep{*}\leftarrow\rep{*}^{(k)},\] where the first map is the evident underlying-object functor $\Alg\mathfrak C\to\mathfrak C$ and the second map is the inclusion. 
\end{defn}
Said another way, $\Alg\rep{*}^{(k)}$ is the category of algebras whose underlying object has level $\leq k$.

An object in $\Alg\rep{*}^{(k)}$ which is simple is also simple in $\Alg\rep{*}$, since quotient objects have \bigness{} less than or equal to the original. Conversely, a simple object in $\Alg\rep{*}$ of level $k$ is still simple when considered as an object of $\Alg\rep{*}^{(k)}$, since a nontrivial quotient of it in $\Alg\rep{*}^{(k)}$ would also exist in $\Alg\rep{*}$. Therefore, we can speak unambiguously of an algebra being simple without worrying about the category that it is considered as an object of.

In the algebro-geometric framework, $\Ind{k}{*}$ is a functor $\Rep{k}\boxtimes\rep{*-k}^{(l)}\to\rep{*}^{(l+k)}$. Here $\rep{*-k}^{(l)}$ is  $\rep{*}^{(l)}$, but with the maps $r : \ob\rep{*}^{(l)}\to\A$ resp.~$r' : \mor\rep{*}^{(l)}\to\A$ replaced by $\tau_{k}r$ resp.~$\tau_{k}r'$, where $\tau_{k}:\A\to\A$ is translation by $k$. 
However, induction is not everywhere defined; it has poles at small integers. We can at least say, though, that for each Zariski-open $U\subseteq \A$ and each section of $\Rep{k}\boxtimes\rep{*-k}^{(l)}$ defined on $U$, there is a Zariski-open $V\subseteq U$ which contains all but at most finitely many points of $U$ and a section of $\rep{*}^{(l+k)}$ defined on $V$. We will only consider the case when $U = \A$, and the sections of $\Rep{k}\boxtimes\rep{*-k}^{(l)}$ will take the form $A \boxtimes \triv$.

Induction canonically carries a lax monoidal structure. This means that when $A$ is an algebra in $\Rep{k}$, the trivial algebra structure on $\triv$ means that $A\boxtimes \triv$ is an algebra in $\mathcal R (\Rep{k}\boxtimes\rep{*-k}^{(l)})$. This lets us make the following definition.
\begin{defn}
If $A$ is an object in $\Rep{k}$, we define $\ind{A}$ to be the section $\Ind{k}{*}(A\boxtimes\triv)$ of $ \rep{*}^{(k)}$. If $A$ is an algebra in $\Rep{k}$, we define $\ind{A}$ to be $\Ind{k}{*}(A\boxtimes \triv)$, together with the induced algebra structure, considered as an section of $\Alg\rep{*}^{(k)}$. We also write $U_A$ for the (Zariski-open, cofinite) subset of $\A$ on which $\ind{A}$ is defined.
\end{defn}

\section{Constructibility}\label{cons}

\subsection{Basics}

Our main tool for proving things about the transcendental case is using constructibility arguments.
\begin{lem}\label{con-lem}
If $V$ is a scheme, and $r$ a map from $V$ to the complex line, and $S$ a constructible subset of $V$, then $r(S)$ is either finite or cofinite.
\end{lem}
\begin{proof}
Chevalley's theorem on constructible sets means that the image $r(S)$ is a constructible subset of $\A$. However, constructible subsets are generated under intersection and complement by (Zariski-)closed subsets, which for $\A$ are just the finite subsets, so constructible subsets of the complex line are the finite or cofinite subsets.
\end{proof}
\begin{cor}\label{con-cor-lem}
With $V$, $r$, and $S$ as above, suppose further that for all $\theta\in\Gal(\mathbb{C}/\overline{\mathbb{Q}})$, there is a map $\Theta : V \to V$ with $\Theta(S)\subseteq S$ and $r\Theta=\theta r$. If $r(S)$ is infinite, it is cofinite and contains all transcendental numbers.
\end{cor}
\begin{proof}
By \autoref{con-lem}, $r(S)$ is cofinite or finite; since it is not finite, it is cofinite. Since there are infinitely many transcendental numbers, at least one is in $r(S)$, say $s$; that is, there is $x\in V$ with $r(x)=s$. However, for all transcendental $t$ there is a Galois automorphism $\theta$ sending $s$ to $t$. Since $r(\Theta(x)) = \theta(r(x)) = \theta(s) = t$, it follows that $t \in r(S)$.
\end{proof}

We will usually apply this when $(V,r)$ is the object scheme-over-$\A$ of a category over $\A$, and $\Theta$ is the object part of a Galois functor.

\subsection{Constructible Properties}

We can mimic the proof of \cite[Proposition 3.4]{mat}, in the non-Ab-enriched case to show the following.

\begin{lem}\label{sim-const}
There is a constructible subset $Sim$ of $\ob\Alg\rep{*}^{(k)}$ whose closed points are simple algebras.
\end{lem}
\begin{proof}
	Put $\ob = \ob\Alg\rep{*}^{(k)}$ and $\mor= \mor \Alg\rep{*}^{(k)}$. Let $z : \ob\to\mor$ be the map which sends an object to the map from that object to the terminal algebra. Since $\Alg\rep{*}^{(k)}$ is of finite type, by \cite[Proposition 3.2]{mat}, there is a constructible set $Mon\subset\mor$ which parameterizes monomorphisms, in that its intersection with each fiber consists of the actual set of monomorphisms in that category. Then the set $\mor-Mon-z(\ob)\subset\mor$ parameterizes morphisms which are not zero or monic. The complement of the image of this set under the domain map $\mor\to\ob$ parameterizes simple objects; we call it $Sim$. By Chevalley's theorem and the basic properties of constructible sets, all subsets just mentioned are constructible.

In the non-Ab-enriched case, we want a slightly stronger definition of simplicity than just ``every map out of it is monic or zero''; we want there to be no nontrivial quotients, up to isomorphism. However, in the semisimple case, the two are equivalent, and it follows from \cite[Theorem 6.10]{cos} that maps in $\rep{t}$ that are both monic and epic are isomorphisms, even in the case of integer $t$. Therefore the set $Sim$ constructed earlier is in fact the set of simple objects, in the stronger sense.
\end{proof}

\begin{lem}
There is a constructible subset $-\cong-$ of \[\ob\Alg\rep{*}^{(k)}\times_{\A}\ob\Alg\rep{*}^{(k)}\] whose closed points correspond to pairs of algebras which are isomorphic to each other.
\end{lem}
\begin{proof}
See \cite[Lemma 6.4]{mat}.
\end{proof}

\begin{lem}
There is a constructible subset $-\hookrightarrow-$ of \[\ob\rep{*}^{(k)}\times_{\A}\ob\rep{*}^{(k)}\] whose closed points correspond to pairs of algebras $(A,B)$ for which there exists an injection from $A$ to $B$.
\end{lem}
\begin{proof} By \cite[Proposition 3.2]{mat}, the set of monomorphisms $Mon$ is constructible; projecting via the map \[\xymatrix{\mor\rep{t}^{(k)}\ar[rr]_-{(\mathrm{dom},\mathrm{cod})}& &\ob\rep{t}^{(k)}\times_{\A}\ob\rep{t}^{(k)}}\] gives the subset $-\hookrightarrow-$ desired.
\end{proof}

\subsection{Substitution}

We formulate a principle of substitution for constructible sets.

\begin{lem}
Let $\xymatrix@1{A \ar[r]_{a} & \A}$ and $\xymatrix@1{B \ar[r]_{b} & \A}$ be two schemes over $\A$. Let $X : U \to B$ be a section of $B$, i.e.~$b X$ is the inclusion $U\hookrightarrow \A$. Let $R$ be a constructible subset of $A\times_{\A} B$. Then there is a constructible subset of $A$ consisting of elements $x$ for which $a(x)$ is in $U$ and $(x, X(a(x)))$ is in $R$.
\end{lem}
\begin{proof}
Write $A_U = A\times_{\A}U$, etc. Then $(A\times_{\A}B)_U \cong A_U \times_U B_U$. The section $X$ can also be pulled back to a map $X_U : U \to B_U$. Let $R'$ be the inverse image of $R$ under the composite
\[\xymatrix{
A_U \cong A_U \times_U U \ar[r]_-{1\times_U X_U} & A_U \times_U B_U \cong (A\times_{\A}B)_U \ar[r] & A\times_{\A} B
}\]
and let $S$ be the image of $R'$ under the map $A_U \to A$. By Chevalley and the fact that inverse images of constructible sets are constructible, $S$ is constructible. It can be checked that $S$ is the subset desired.
\end{proof}
In the situation above, we call the subset $S$ the \emph{substitution} of $X$ into $R$. We also use notation that matches this; for example, if $R$ is the subset of $\ob\Alg\rep{*}^{(k)}\times_{\A}\ob\Alg\rep{*}^{(k)}$ consisting of pairs of isomorphic algebras, which we write as $-\cong-$, and $X$ is $\ind{A}$, we write the substitution of $X$ into $R$ as $-\cong \ind{A}$.
 
\begin{lem}
If $A$ is a simple algebra in $\Rep{k}$, then $\ind{A}(t)$ is a simple algebra in $\rep{t}$ for cofinitely many $t$, including all transcendental $t$.
\end{lem}
\begin{proof}
The set of $t$ for which the statement holds is constructible, because it is the substitution of $\ind{A}$ into $Sim$. Hence if it is infinite, the desired conclusion follows by \autoref{con-cor-lem}. $V$ is $\ob\Alg\rep{*}^{(k)}$. The $\Theta$ which lifts a given automorphism $\theta$ is the object part of the extension of $\theta_*$ to algebras.

We will show that the intersection is infinite by showing that it contains all sufficiently large natural numbers. If $n>2k$, then $\ind{A}(n)\in\Alg\rep{n}$ corresponds to $\Ind{k}{n}(A\boxtimes\triv)\in\Rep{n}$, where $A$ becomes a representation of $S_k\times S_{n-k}$ by $S_{n-k}$ acting trivially. These are simple, by \autoref{ind-simp}. Simplicity in $\Rep{n}$ implies simplicity in $\rep{n}^{(k)}$ for sufficiently large $n$ because any nontrivial quotient that exists in $\rep{n}^{(k)}$ would have level not greater than the level of $\ind{A}(n)$, which is $k$, hence would not be killed by the projection to $\Rep{n}$. Hence for $n$ large enough, $\ind{A}(n)$ is also simple and thus belongs to  $Sim$. 
\end{proof}

In fact, $\ind{A}(t)$ is simple for all $t$ for which it is defined, for categorical reasons. However, the statement here is sufficient for our purposes, and the proof of the stronger statement would be an unnecessary detour.

\begin{lem}\label{ind-iso}
For all algebras $A$ in $\Rep{k}$, the subset of those $X\in\ob\Alg\rep{*}^{(k)}$ for which there exists an isomorphism from $X$ to $\ind{A}(t)$, for $t=$ rank of $X$, is constructible. \end{lem}
\begin{proof}
This is because it is the substitution of $\ind{A}$ into $-\cong-$.
\end{proof}

\begin{lem}\label{ind-mon}
For all objects $B$ in $\Rep{k}$, the subset of those $A\in\ob\Alg\rep{*}^{(k)}$ for which there exists an injection of objects from $A$ to $\ind{B}(t)$, for $t=$ rank of $A$, is constructible. \end{lem}
\begin{proof}
\sloppy
This is because it is the inverse image under the underlying-object map $\ob\Alg\rep{*}^{(k)}\to\ob\rep{*}^{(k)}$ of the substitution of $\ind{B}$ into \mbox{$-\hookrightarrow-$}.
\end{proof}
\fussy

\section{Classification}\label{sec:main}
In fact, for $t$ transcendental, the only simple algebras in $\rep{t}$ come from algebras in $\Rep{k}$, for some $k$.

\begin{defn}
Suppose $k\in\mathbb N$ and $B \in \Rep{k}$. A number $t\in\C$ is called \emph{$(k,B)$-exotic} if there is a simple algebra $A$ in $\rep{t}$ such that 
\begin{itemize}
\item $A$ has \bigness{} less than or equal to $k$,
\item there is an injection of objects $A\hookrightarrow\ind{B}(t)$, and
\item  $A$ is \emph{not} isomorphic as an algebra to $\ind{C}(t)$, for any $k'$ and simple algebra $C\in\Rep{k'}$.
\end{itemize}
If $t$ is not $(k,B)$-exotic, it is called \emph{$(k,B)$-ordinary}.
\end{defn}

\begin{lem}
For a given $k$ and $B$, the set of $(k,B)$-exotic numbers, and hence also the set of $(k,B)$-ordinary numbers, is constructible.
\end{lem}
\begin{proof}
The set of $(k,B)$-exotic numbers is the projection from $\Alg\rep{t}^{(k)}$ to $\C$ of the subset consisting of algebras which are simple, inject into $\ind{B}(t)$, and not isomorphic to $\ind{A}(t)$ for any $k'\leq k$ and $A\in\Rep{k'}$. It suffices to consider $k'\leq k$ because algebras induced from $\Rep{k'}$ for $k'>k$ will have \bigness{} greater than $k$. Since there are finitely many such $A$, due to \autoref{actual-classification}, this is the intersection of finitely many constructible sets: the set of simple algebras is constructible by \autoref{sim-const}, the set of algebras whose underlying objects inject into $\ind{B}(t)$ is constructible by \autoref{ind-mon}, and for each $A$ the set of algebras not isomorphic to $\ind{A}(t)$ is constructible by \autoref{ind-iso}. The conclusion follows.
\end{proof}

\begin{lem}\label{large-ordinary}
For all $k$ and $B$, if $n$ is sufficiently large, then $n$ is $(k,B)$-ordinary.
\end{lem}
\begin{proof}
If $n > 2k$, then the projection functor $\rep{n} \to \Rep{n}$ is an equivalence onto its image, by \autoref{proj-equiv}. Let $A$ be a simple algebra in $\rep{n}$, and $\tilde{A}$ its image under the projection in $\Rep{n}$. The algebra $\tilde A$ is simple, because any quotient in $\Rep{n}$ would necessarily be in the image of the projection, hence would imply the existence of a quotient in $\rep{n}$, contradicting the simplicity of $A$.  For sufficiently large $n$, it follows from \autoref{level-class} that $\tilde{A}$ is isomorphic to the induction of some simple algebra $B$ from $\Rep{k'}$ for some $k'\leq k$. Then $A\cong \ind{B}(t)$, as desired.
\end{proof}

\begin{thm}\label{main}
For all $k$ and $B$, the set of $(k,B)$-ordinary numbers is cofinite and includes all transcendentals.
\end{thm}
\begin{proof}
This follows from \autoref{large-ordinary} and \autoref{con-cor-lem}. Since for each $k$ and $B$, the set of $(k,B)$-ordinary numbers is constructible and contains all sufficiently large natural numbers, it is cofinite and includes all transcendentals.
\end{proof}

\begin{thm}\label{main-2}
For all transcendental $t$ and simple commutative algebras $A\in\rep{t}$, there is a $k$ and simple commutative algebra $B\in\Rep{k}$ with $A\cong\Ind{k}{t}(B\boxtimes\triv)$, as algebras.
\end{thm}
\begin{proof}
This is immediate from \autoref{main}. Suppose that $A$ is of level $k$; then there is a $B'\in\Rep{k}$ such that $A\hookrightarrow\ind{B'}(t)$. Then since $t$ is $(k,B')$-ordinary, the conclusion follows.
\end{proof}

The arguments of this paper used surprisingly few properties specific to simple commutative algebras. It was important that the induction of an algebra was canonically an algebra, and that there were only finitely many algebras in each $\Rep{k'}$. The constructibility methods of this paper should therefore, in priniciple, allow a similar classification of simple associative or Lie algebras in $\rep{t}$ as interpolations of families of corresponding structures in $\Rep{n}$.

\appendix

\section{Simple Commutative Algebras in $\Rep{n}$}\label{app}
Here, we classify the simple commutative algebras internal to the actual categories of representations of the symmetric group. Since representations of $S_n$ are vector spaces with extra structure, a simple algebra turns out to be an ordinary algebra over $\C$ with a compatible $S_n$-action, such that no nontrivial ideal is closed under the action.
\begin{lem}
If $X$ is a set with a transitive action of $S_n$, then $\C[X]$ with the pointwise algebra structure is simple.
\end{lem}
\begin{proof}
By ``the pointwise algebra structure'', we mean that $\C[X]$ is considered as the algebra of functions $X\to\C$ with pointwise multiplication. The $S_n$-action is given by $(\sigma f)(x) = f(\sigma^{-1}x)$. If $x\in X$, write $\tilde x$ for the indicator function of $x$: $\tilde x(y) = 1$ if $x = y$, and $0$ else. The set of indicator functions forms a basis. Now, suppose that there is an ideal closed under $S_n$. If it contains a nonzero element, say $c_0 \tilde x_0 + c_1 \tilde x_1+\ldots$ with $c_0 \neq 0$, it then contains $(\frac1{c_0} \tilde x_0) (c_0 \tilde x_0 + \ldots) = \tilde x_0$. Then, for all $x$, there is $\sigma\in S_n$ taking $x_0$ to $x$, and $\sigma \tilde{x}_0 = \widetilde{\sigma x_0} = \tilde x$, so all basis elements $\tilde x$ must be in the ideal, which is therefore the entire algebra.
\end{proof}

Any subgroup $H$ of $S_n$ gives an transitive $S_n$-action on $S_n/H$, the cosets of $H$ with the action of left multiplication. Two subgroups give isomorphic $S_n$-actions iff they are conjugate. Every transitive $S_n$-action is of this form; given such an action, $H$ can be recovered as the stabilizer of any element.

\begin{lem}\label{actual-classification}

Every simple algebra in $\Rep{k}$ is of the form $\C[S_n/H]$. 

\end{lem}
\begin{proof}

If $A$ is a simple algebra in $\Rep{n}$, consider the spectrum of the underlying ring, $\operatorname{Spec} A$. Since $A$ is a finite dimensional complex algebra, $\Spec A$ is finite. Also, $\Spec A$ is reduced, because the ideal of nilpotent elements is fixed by $S_n$ (if $a^k = 0$, then $(\sigma \cdot a)^k = \sigma \cdot a^k = 0$) and is therefore the zero ideal. It then follows that global sections of the structure sheaf on $\Spec A$  correspond exactly with continuous functions from the discrete space of closed points of $\Spec A$ to $\C$, and the algebra structures also correspond. Thus $A \cong \C[\Spec A]$.

There is an action of $S_n$ on $\Spec A$, defined by $\sigma \cdot p := \{\sigma\cdot  a | a \in p\}$. This action commutes with the above isomorphism. The action must be transitive, since if not, consider a nontrivial orbit $\mathcal O$. The intersection of $\mathcal O$ is an ideal fixed by $S_n$. It is proper, since all elements of $\mathcal O$ are. It is nonzero, because if the intersection of the elements of $\mathcal O$ were zero, any element of $\C[\Spec A]$ vanishing on $\mathcal O$ would necessarily be zero, which would imply $\mathcal O = \Spec A$. The result then follows from the discussion before the proposition.
\end{proof}

The above holds, \emph{mutatis mutandis}, for an arbitrary finite group $G$ in place of $S_n$: the simple algebras in $\mathbf{Rep}(G)$ are precisely $\C[G/H]$ for $H$ a subgroup of $G$.

From this classification, we obtain the following.
\begin{lem}\label{ind-simp}
If $A$ is a simple algebra in $\Rep{k}$, the induced algebra structure on $\Ind{k}{n}(A\boxtimes \triv)$ is simple.
\end{lem}
\begin{proof}
If $A \cong \C[S_k/H]$, then $\Ind{k}{n}(A\boxtimes\triv) \cong \C[S_n/H\times S_{n-k}]$ as algebras.
\end{proof}

We require the following group-theoretic fact to establish that algebras of level $k$ are induced from $\Rep{k'}$ for $k'\leq k$.

\begin{lem}\label{contains-times}
For $n > 2k+1$, if $H$ is a subgroup of $S_n$ containing $S_{n-k}$, it is conjugate to a subgroup of the form $H'\times S_{n-k'}$ for some $k'\leq k$ and subgroup $H$ of $S_{k'}$.
\end{lem}
\begin{proof}
Suppose $H$ is not contained in $S_k\times S_{n-k}\subset S_n$, for if it does not we are finished. Suppose that (with the standard action of $S_n$ on the finite set $[n]$) the size of the orbit under $H$ of $n$ is $j + n - k$, where $j>0$. Without loss of generality, we can pass to a conjugate $H'$ and suppose that the elements of the orbit consist of the last $n-(k-j)$, in particular including the last $n-k$.
We claim that $H'$ contains $S_{n-(k-j)}$, acting on the last $n-(k-j)$ elements.
It suffices to show that all two-cycles swapping any two of the last $n-(k-j)$ elements belong to $H'$; there are three possible cases.

 If both elements are among the last $n-k$, then the two-cycle $\alpha$ swapping them belongs to $S_{n-k}$ and hence to $H'$.

If one is in the middle $j$, and the other in the last $n-k$, call the former $x$ and the latter $y$. Suppose $\sigma(x) = n$ for some $\sigma\in H'$. Since $n - k > k + 1 \geq j + 1$, by the pigeonhole principle, there exists a $z$ among the last $n-k$ which is sent under $\sigma$ to neither any of the middle $j$ elements nor the last element $n$. Let $\alpha \in H'$ be the two-cycle swapping $y$ and $z$, and let $\beta \in H'$ be the two-cycle swapping $\sigma(z)$ and $n = \sigma(x)$. Then we can swap $x$ and $y$ via $\alpha\sigma^{-1}\beta\sigma\alpha\in H'$.

Finally, if both elements are among the middle $j$, suppose that they are $x$ and $x'$. Let $\sigma(x) = n$ for $\sigma\in H'$. By the previous two cases, we can swap $n$ and $\sigma(x')$ no matter what $\sigma(x')$ is. Then we can swap $x$ and $x'$ via $\sigma^{-1}\alpha\sigma\in H'$, where $\alpha\in H'$ is the two-cycle swapping $n$ and $\sigma(x')$.

Now, $H'$ is contained in in $S_{k-j}\times S_{n-(k-j)}$, because the last $n-(k-j)$ elements are all in one orbit which contains none of the first $k-j$. Then $H'$ is necessarily of the form $H''\times S_{n-k'}$, for $H''$ a subgroup of $S_{k'}$. 
\end{proof}

\begin{lem}\label{level-class}
For all objects $B\in\Rep{k}$, for large enough $n$, any simple algebra in $\Rep{n}$ which is in the image of the projection from $\rep{n}$ and injects into $\Ind{k}{n}(B\boxtimes\triv)$ as objects is isomorphic to an algebra induced from $\Rep{k'}\boxtimes\Rep{n-k'}$, with the $\Rep{n-k'}$ part trivial, for some $k'\leq k$.
\end{lem}
\begin{proof}
Suppose our simple algebra is $\C[S_n/H]$. Looking at the dimensions of the underlying vector spaces of $\C[S_n/H]$ and $\Ind{k}{n}(B\boxtimes\triv)$, we have $[S_n : H] = \frac{n!}{|H|} \leq {n \choose k} \dim B$. We are interested in the index of $H\cap A_n$ in $A_n$. There are two possibilities: either $H = H\cap A_n$, or $H = (H\cap A_n)\cup h(H\cap A_n)$ for $H\ni h \notin A_n$. 

The first case is impossible, for the following reason. If $H\subseteq A_n$, then $\C[S_n/H] \cong \mathrm{Ind}_H^{S_n}(\triv) \cong \mathrm{Ind}_{A_n}^{S_n}(\mathrm{Ind}_H^{A_n}(\triv))$. Since induction is lax monoidal, there is a natural map $\C\to\mathrm{Ind}_H^{A_n}(\triv)$ which can be shown to be injective. Then, inducing this trivial representation from $A_n$ to $S_n$ means that $\C[S_n/H]$ will have the sign representation as a summand. But this is not possible if $\C[S_n/H]$ is in the image of the projection from $\rep{n}$.

So, we know that $|H\cap A_n| = |H|/2$, so $[A_n : H\cap A_n]$ = $\frac{|S_n|/2}{|H|/2} = [S_n : H]$. If $n$ is sufficiently large, ${n \choose k+1} = {n \choose k}\frac{n-k}{k+1}$ is greater than ${n \choose k}\dim B$, so $[A_n : H\cap A_n] = [S_n : H] \leq {n \choose k+1}$. Using \cite[Theorem 5.2A]{pgrp}, we conclude that for some $j\leq k$ and $\sigma \in S_n$, there is an inclusion $\sigma A_{n-j} \sigma^{-1}\subseteq H\cap A_n$. Here, $A_{n-j}$ is considered as a subgroup of $S_n$ as the group of even permutations of the last $n-j$ objects out of the $n$ objects that $S_n$ permutes. Since $A_n$ is normal in $S_n$, and $H$ was only determined up to conjugacy to begin with, we can assume without loss of generality that $\sigma$ is the identity. We know that there exists $H\ni h\notin A_n$; for such an element, $A_{n-j}\cup hA_{n-j} \subseteq H$. It then follows that $S_{n-j}\subseteq H$.

Finally, since for sufficiently large $n$ we can apply \autoref{contains-times}, we conclude that $H$ is conjugate to a subgroup of the form $H'\times S_{n-j'}$ for some $j'\leq j \leq k$. Then our original algebra $\C[S_n/H]$ is isomorphic to $\Ind{j'}{n}(\C[S_{j'}/H']\boxtimes\triv)$, as desired.

\end{proof}

\bibliographystyle{alpha}
\bibliography{biblio}
\end{document}